\documentclass[12pt]{amsart}
\usepackage{comment}
\usepackage{color}
\usepackage{amsmath,amssymb,amsthm}
\usepackage{tikz}
\usepackage{pgfplots}
\usetikzlibrary{shapes.geometric, lindenmayersystems, decorations.pathmorphing}
\usetikzlibrary{calc}
\pgfplotsset{compat=newest}
\makeatletter
\@namedef{subjclassname@2020}{%
	\textup{2020} Mathematics Subject Classification}
\makeatother

\UseRawInputEncoding

\newtheorem{theorem}{Theorem}[section]
\newtheorem{lemma}[theorem]{Lemma}
\newtheorem{proposition}[theorem]{Proposition}
\newtheorem{corollary}[theorem]{Corollary}

\theoremstyle{definition}
\newtheorem{definition}[theorem]{Definition}
\newtheorem{example}[theorem]{Example}

\theoremstyle{remark}
\newtheorem{remark}[theorem]{Remark}
\numberwithin{equation}{section}

\def\C{\mathbb C}

\def\D{\mathbb D}

\def\ov{\overline}

\def\Ha{\mathbb H}

\def\R{\mathbb R}

\renewcommand{\Re}{\operatorname{Re}}
\renewcommand{\Im}{\operatorname{Im}}
\begin{document}

\numberwithin{equation}{section}
	
	\title[Failure of the Denjoy-Wolff Theorem]
	{On the failure of the Denjoy-Wolff Theorem in convex domains}

	\author{Filippo Bracci}
	\author{Ahmed Yekta \"Okten}  
	
	\address{F.Bracci, A. Y. \"Okten:  Dipartimento di Matematica, Universit\`a Di Roma ``Tor Vergata''€, Via Della Ricerca Scientifica 1, 00133 Roma, Italy}

	\email{fbracci@mat.uniroma2.it}

	\email{okten@mat.uniroma2.it}

	\thanks{Partially supported by the MIUR Excellence Department Project 2023-2027 MatMod@Tov awarded to the Department of Mathematics, University of Rome Tor Vergata, by PRIN (2022) Real and Complex Manifolds: Geometry and holomorphic dynamics Ref: 2022AP8HZ9 and by GNSAGA of INdAM}

	\begin{abstract}  
In this note, we construct examples of bounded smooth convex domains with no non-trivial analytic discs on the boundary which possess a holomorphic self-map without fixed points so that the iterates do not converge to a point (that is, the Denjoy-Wolff theorem does not hold). We also show that, in the case of bounded convex domains with $C^{1+\varepsilon}$-smooth boundary which have non-trivial analytic discs on the boundary, the cluster set of the orbits of holomorphic self-maps without fixed points can be equal to the principal part of any prime end of any planar bounded simply connected domain.		
 \end{abstract}

	
	
	\maketitle

\section{Introduction}

Let $\Omega\subset \C^N$ be a bounded domain, $N\geq 1$. Let $F\in \mathcal{O}(\Omega,\Omega)$ be a holomorphic self-map of $\Omega$. The {\sl target set of $F$} is the set 
\begin{equation*}
\begin{split}
T(F):=\{z\in \partial{\Omega}: \exists z_0 \in \Omega,  \{n_k\}\subset \mathbb N \text{\ such that} \:\: \lim_{k\to \infty}n_k=\infty, \lim_{k\to\infty}F^{\circ n_k}(z_0)= z\}.
\end{split}
\end{equation*}

\begin{definition}\label{def:wolffdenjoyproperty}
We say that a bounded domain $\Omega$ in $\C^N$, $N\geq 1$, satisfies \emph{the Denjoy-Wolff property} if for any holomorphic self-map $F$ of $\Omega$ either $T(F)=\emptyset$ or  there exists a point $p\in\partial\Omega$ such that the sequence of iterates $\{F^{\circ n}\}$ of $F$ converges uniformly on compacta to the constant map $z\mapsto p$, that is $T(F)=\{p\}$. 
\end{definition}

It is a well known and classical result in complex analysis that the unit disc $\D$ in $\mathbb{C}$ satisfies the Denjoy-Wolff property. This fact was independently proven by Denjoy and Wolff \cite{Den, Wol1,Wol2} at the beginning of the last century. Since then, it has been a fundamental tool in the iteration theory of holomorphic maps and it has been extended in many different direction by several mathematicians (see, {\sl e.g.}, \cite{AbateTaut}, and \cite{Kar}  in the context of  distance non-expansive maps of Gromov hyperbolic metric spaces).  

If $\Omega$ is convex, Abate (see, \cite[Theorem~2.4.20]{AbateTaut}) proved that if $f\in \mathcal{O}(\Omega,\Omega)$ then   $T(F)=\emptyset$ if and only if $f$ has fixed points in $\Omega$. Thus, for a bounded convex domain $\Omega$  the Denjoy-Wolff property holds if and only if $T(F)$ contains only one point for every holomorphic self-map $F$ of $\Omega$ without fixed points.

The first generalization to higher dimensions of the classical Denjoy-Wolff theorem appeared in the second half of the last century, when Herv\'e \cite{Her2} proved that the unit ball $\mathbb B^N$ of $\mathbb C^N$ verifies the Denjoy-Wolff property. Abate in \cite{Aba, AbateTaut} extended the result of Herv\'e to strongly convex domains with $C^2$-smooth boundaries. Later Budzy\'nska \cite{Bud} (see also \cite{AR}) removed the boundary regularity assumptions and extended this result to strictly convex domains, while in \cite{BGZ} the same result has been proved for bounded convex domains which are Gromov hyperbolic with respect to the Kobayashi distance---such a class includes, for instance, smooth D'Angelo finite type bounded convex domains (see \cite{Zi1, Zi2}). See also \cite{BKR, CR} and references therein for relevant results about the Denjoy-Wolff theorem in complex Banach spaces and in  symmetric domains.

However, the basic feature of Gromov hyperbolicity which allows the Denjoy-Wolff theorem is the ``bending inside'' property of geodesics. Based on such observation, recently, Bharali and Zimmer \cite{BZ} (see also \cite{BNT}) introduced the notion of ``visibility'' and Bharali and Maitra \cite[Theorem 1.8]{BM} proved if $\Omega$ is a visible bounded convex domain in $\C^N$ then it satisfies the Denjoy-Wolff property. 

However, in general, visibility is a stronger  condition that the Denjoy-Wolff property: in \cite[Corollary 1.6]{BB} it is proved that there exist  bounded simply connected  domains in $\C$  which are not visible but for which the Denjoy-Wolff property  holds. 

On the other hand, the only known examples in the literature of  convex domains not having the Denjoy-Wolff property are product domains:  if $D_1\subset \C^{N_1}$, $D_2\subset \C^{N_2}$ are two bounded convex domains and $F\in  \mathcal{O}(D_1,D_1)$ is any holomorphic self-map without fixed points, then the holomorphic self-map $D_1\times D_2\ni (z,w)\mapsto (f(z), w)$ of the bounded convex domain $\Omega:=D_1\times D_2$ does not have fixed points in $\Omega$ and its target set contains $\{p\}\times D_2$ for some $p\in \partial D_1$. In particular, $\Omega$ does not have the Denjoy-Wolff property. Even in this case however, it is not clear what the target set could really be. We mention here that in the bidisc $\D^2$,  Herv\'e \cite{Her} proved that the target set of a holomorphic self-map of $\D^2$ has to be contained in either $\{e^{i\theta}\}\times \overline{\mathbb D}$ or in $\overline{\mathbb D}\times \{e^{i\theta}\}$ (for some $\theta\in \R$). So, apparently, the existence of non-trivial analytic discs on the boundary is the primary obstruction of the Denjoy-Wolff property. 

The main aim of this paper is to deal with  (non-trivial) constructions of holomorphic self-maps of   bounded convex domains for which the Denjoy-Wolff property fails and to examine the possible target set in case the domain contains a non-trivial analytic disc in the boundary. Indeed it turns out that there exist bounded convex domains with smooth boundaries containing no non-trivial analytic discs for which the Denjoy-Wolff property does not hold:

\begin{theorem}\label{cor:existence-nonDW}
There exists a bounded convex domain  $\Omega\subset \C^2$,  such that $\partial \Omega$ is $C^\infty$-smooth,  $\partial \Omega$ does not contain  non-trivial analytic discs  and $\Omega$ does not have  the Denjoy-Wolff property. 
\end{theorem}

 {Our construction, in fact, shows that there exist a complex affine line $L$ and a holomorphic self-map $F$ of $\Omega$ such that $\partial \Omega\cap L$ is a real segment of dimension one and $T(F)$ is a (non-trivial) real segment contained in  $\partial \Omega\cap L$}. 

In case of bounded convex domains with non-trivial analytic discs on the boundary, we can also prove the following:

\begin{theorem}\label{prop:target-set-discs}
Let $\Omega\subset\C^N$, $N\geq 2$ be a bounded convex domain with $C^{1+\varepsilon}$-smooth boundary. Suppose that $\partial \Omega$ contains a complex affine disc {$\Delta$}. Let $K$ be the principal part of any prime end of any bounded simply connected domain in $\C$. Then there exist a holomorphic self-map $F$ of $\Omega$ without fixed points such that, up to homotheties, $T(F)=K \subset \Delta$.
 \end{theorem}
 
In particular, $T(F)$ can be as ``bad'' as the principal part of any bounded simply connected domain, for instance, it can have non-integer Hausdorff dimension. 
 
{If $D\subset \C^N$ is a bounded convex domain and $\partial D$ contains a non-trivial analytic disc, then convexity implies that $\partial D$ also contains a complex affine disc. Therefore}, it follows from the previous theorem that if $\Omega\subset \C^N$, $N\geq 2$, is  a bounded convex domain with $C^{1+\varepsilon}$-smooth boundary for some $\varepsilon>0$ and satisfies the Denjoy-Wolff property,  then $\partial \Omega$ does not contain non-trivial analytic discs. 

Existence of  non-trivial analytic discs in the boundary of convex domains prevents visibility. Also, in \cite[Section 5]{BNT} and \cite[Section 5]{Okt} there are examples of smooth bounded convex domains with no non-trivial analytic discs on the boundary such that the order of contact of the boundary  with a complex affine  line  is ``very high'' and do not enjoy visibility. These similitudes between the {known} conditions for a bounded convex domain to be visible and for satisfying the Denjoy-Wolff property suggest the following conjecture, which is open even in the smooth case:

\smallskip

{\bf Conjecture:} {\sl Let $\Omega\subset \C^N$ be a bounded convex domain, $N\geq 2$. Then $\Omega$ is visible if and only if it satisfies the Denjoy-Wolff property}.

\smallskip

The proofs of our results are based on the following construction. Let $\Ha:=\{\zeta\in \C: \Re \zeta>0\}$. Let $V\subset \Ha$ be  a convex domain such that $0\in \partial V$ and $\partial V$ is $C^{1+\varepsilon}$ at $0$ for some $\varepsilon>0$. Let $\Psi:(0,\delta)\to (0,+\infty)$ be a convex strictly increasing function such that $\lim_{x\to 0^+} \Psi(x)=0$. The {\sl wedge} $W(V, \Psi,\delta)$ is defined as
\[
W(V,\Psi, \delta):=\{(z_1,z_2)\in\C^2: \|(z_1, z_2)\|<\delta, z_1 \in V, \Re z_2>0, \Re z_1>\Psi(\Re z_2)\}.  
\]
The set $W(V, \Psi,\delta)$ is a bounded convex domain in $\C^2$ and $H:=\{(z_1, z_2)\in \C^2:  z_1=0\}$ is a complex supporting hyperplane for $W(V, \Psi,\delta)$ such that 
\[
I_{V, \Psi,\delta}:=\partial W(V, \Psi, \delta)\cap H=\{(0, it):  t\in [-\delta,\delta]\}.
\]
    {For $z=(z_1,\ldots, z_N)\in \C^N$, $N\geq 2$, we denote by $z''=(z_3,\ldots, z_N)$}. We have:
\begin{theorem}\label{thm:main}
{For every $\delta>0$ and for every $V\subset \Ha$  convex domain such that $0\in \partial V$ and $\partial V$ is $C^{1+\varepsilon}$-smooth at $0$ for some $\varepsilon>0$, there exist $\delta'\in(0,\delta)$ and a  convex strictly increasing function  $\Psi_0:(0,+\infty)\to (0,+\infty)$ such that $\lim_{x\to 0^+} \Psi_0(x)=0$, with the following property. If $\Omega\subset \C^N$, $N\geq 2$, is any  domain so that  $W(V, \Psi_0, \delta)\times\{0''\}\subset \Omega \subset \{z\in \mathbb{C}^N: \Re z_1> 0\}$ and $I_{V,\Psi_0,\delta}\times\{0''\}\subset \partial\Omega$ then there exists a holomorphic self-map $F$ of $\Omega$ without fixed points in $\Omega$ so that $T(F)=\{(0, it):  t\in [-\delta',\delta']\}$. In particular $\Omega$ does not satisfy the Denjoy-Wolff property.}
\end{theorem}

{Note that in the previous result, $\Omega$ is not assumed to be bounded, nor convex---although a certain ``weak convexity'' is assumed at the origin. Moreover, no regularity conditions are assumed on $\partial \Omega$. Clearly, the statement is invariant under biholomorphisms in the following sense: if $D\subset \C^N$, $N\geq 2$, is a domain such that there exists a domain $\Omega$ as in the hypotheses of Theorem~\ref{thm:main} and a biholomorphism $\Theta:\Omega\to D$ which extends continuously to $\partial \Omega$ in a neighboorhood of the origin, then $D$ does not satisfy the Denjoy-Wolff property.} 

{The proof of the previous theorem relies on a two-dimensional construction  based on pathological behaviour of one dimensional Riemann maps.  This argument  needs the following  generalisation of a theorem of Carath\'eodory which might be interesting in its own}.

\begin{proposition}\label{Prop:Lehto-NT-simpl}
Let $g:\D\to \C$ be univalent. Let $\sigma\in \partial \D$ and let $\{w_n\}\subset\D$ be a sequence converging non-tangentially to $\sigma$ such that there exists $C>0$ so that $k_\D(w_n,w_{n+1})\leq C$ for all $n$. Then
\[
\Pi[\hat{g}(\underline{\sigma})]=\Gamma(g;\{w_n\}),
\]
where $\Pi[\hat{g}(\underline{\sigma})]$ is the principal part of the prime end of $g(\D)$ corresponding to $\sigma$ and $\Gamma(g;\{w_n\})$ denotes the cluster set of $\{g(w_n)\}$.
\end{proposition}

The paper is organized as follows. In Section~\ref{Sec:uno} we prove Proposition~\ref{Prop:Lehto-NT-simpl} and some preliminary results. In Section~\ref{Sec:3} we prove Theorem~\ref{thm:main} and, as corollary, Theorem~\ref{cor:existence-nonDW}. In Section~\ref{Sec:4} we give the proof of Theorem~\ref{prop:target-set-discs} and construct some explicit examples.

\medskip

The authors sincerely thank the referee for their useful comments and suggestions, which improved the original manuscript.

\section{Preliminary results}\label{Sec:uno}

Let $\Omega$ be a domain in $\mathbb C^n$. Recall that the Kobayashi (pseudo)distance $k_\Omega$ is the inner distance associated to the Kobayashi-Royden pseudometric 
$$ \kappa_\Omega(z;v):=\inf_{\lambda \in \mathbb C}\{|\lambda|: \exists f \in \mathcal{O}(\D,\Omega) \:\:\: \text{such that} \:\:\: f(0)=z, f'(0)\lambda=v\}.$$ The important property of the Kobayashi distance is that it is contracted by holomorphic maps, in particular it is invariant under biholomorphisms. For further details we refer the reader to the standard book \cite{JP}.

Throughout this note, we set $\Ha:=\{\zeta\in\C: \Re \zeta>0\}$ to be the right half plane, while $\C_\infty=\C\cup \{\infty\}$ denotes the Riemann sphere.

The classical Leht\"o-Virtanen's theorem {(see, {\sl e.g.}, \cite[Theorem 3.3.1]{BCDM})} allows to get information on the non-tangential limit of a univalent map of $\D$ by knowing the limit along a curve. Such a result has been generalized   for sequences in \cite[Lemma 5.4]{BB}  (based on an argument from \cite[Theorem 1.5]{AB}). In the following we need a more refined version of \cite[Lemma 5.4]{BB}. As a matter of notations, if $g:\D\to \C$ is a holomorphic map, and $\sigma\in\partial\D$, let
\begin{equation*}
\begin{split}
\Gamma_{R}(g;\sigma):=&\{p\in \C_\infty: \hbox{there exists a sequence $\{r_n\}\subset(0,1)$  converging  to $1$}\\&\hbox{ so that} \lim_{n\to \infty}g(r_n\sigma)=p\}.
\end{split}
\end{equation*}
In other words, $\Gamma_{R}(g;\sigma)$ is the radial cluster set of $g$ at $\sigma$.  Also, we let
\begin{equation*}
\begin{split}
\Gamma_{N}(g;\sigma):=&\{p\in \C_\infty: \hbox{there exists a sequence $\{\zeta_n\}\subset\D$  converging non-tangentially to $\sigma$}\\&\hbox{ so that} \lim_{n\to \infty}g(z_n)=p\}.
\end{split}
\end{equation*}
That is, $\Gamma_{N}(g;\sigma)$ is the non-tangential cluster set of $g$ at $\sigma$. Moreover, if $\{w_n\}\subset\D$ is a sequence converging to $\sigma$, we let 
\[
\Gamma(g;\{w_n\}):=\{p\in \C_\infty: \hbox{there exists a subsequence $\{w_{n_k}\}\subset\D$ so that} \lim_{k\to \infty}g(w_{n_k})=p\},
\]
that is, $\Gamma(g;\{w_n\})$ is the cluster set of the sequence $\{g(w_n)\}$. 

We also need to recall the concept of  principal part of a prime end of a simply connected domain in $\C$ (see, {\sl e.g.}, \cite[Chapter~4]{BCDM}). Let $\Omega\subsetneq \C$ be a simply connected domain and let $g:\D\to \Omega$ be a Riemann map. The map $g$ defines  a one-to-one correspondence between the points of $\partial\D$ and the prime ends of $\Omega$, which are defined in terms of equivalence classes of null-chains $(C_n)$. If $\sigma\in \partial \D$, we denote by $\hat{g}(\underline{\sigma})$ the corresponding prime end in $\Omega$. Hence, the {\sl principal part} of $\hat{g}(\underline{\sigma})$ is defined as
\begin{equation*}
\begin{split}
 \Pi[\hat{g}(\underline{\sigma})]:=&\{\zeta\in\C_\infty:\hbox{there exist a null-chain $(C_n)$ representing $\hat{g}(\underline{\sigma})$}\\& \hbox{and a sequence $\{\zeta_n\}$ such that } \zeta_n\in C_n \hbox{ and} \lim_{n\to \infty}\zeta_n=\zeta\}. 
\end{split}
\end{equation*}

\begin{proposition}\label{Prop:Lehto-NT}
Let $g:\D\to \C$ be univalent. Let $\sigma\in \partial \D$ and let $\{w_n\}\subset\D$ be a sequence converging non-tangentially to $\sigma$ such that there exists $C>0$ so that $k_\D(w_n,w_{n+1})\leq C$ for all $n$. Then
\[
\Pi[\hat{g}(\underline{\sigma})]=\Gamma_{R}(g;\sigma)=\Gamma_{N}(g;\sigma)=\Gamma(g;\{w_n\}).
\]
\end{proposition}
\begin{proof}
It is a result due to Carath\'eodory that $\Pi[\hat{g}(\underline{\sigma})]=\Gamma_{R}(g;\sigma)=\Gamma_{N}(g;\sigma)$ (see, {\sl e.g.}, \cite[Theorem~4.4.9]{BCDM}). Actually, the equality $\Gamma_{R}(g;\sigma)=\Gamma_{N}(g;\sigma)$  can be inferred directly by our proof. Clearly $\Gamma(g;\{w_n\})\subseteq \Gamma_{N}(g;\sigma)$. Thus we have to prove the converse. Without loss of generality, conjugating with a rotation if necessary, we can assume $\sigma=1$. Let $p\in \Gamma_{N}(g;1)$ and let $\{\zeta_n\}$ be a sequence converging non-tangentially to $1$ so that $\lim_{n\to\infty}g(\zeta_n)=p$. This means (see, {\sl e.g.}, \cite[Corollary~6.2.6]{BCDM}) that there is $C_0>0$ such that for every $n\in\mathbb N$ there exists $r_n\in (0,1)$ so that
\[
k_\D(\zeta_n, r_n)\leq C_0.
\]
Note that, since $\{\zeta_n\}$ converges to $1$ this implies that $\{r_n\}$ converges to $1$ as well.

Let $\Omega:=g(\D)$. Since 
\[
k_\Omega(g(\zeta_n), g(r_n))=k_\D(\zeta_n, r_n)\leq C_0,
\]
it follows by the ``Distance Lemma'' (see, {\sl e.g.}, \cite[Theorem~5.3.1]{BCDM}) that $\lim_{n\to\infty} g(r_n)=p$ (note that in particular this implies directly  $\Gamma_{R}(g;\sigma)=\Gamma_{N}(g;\sigma)$). 

Since $\{w_n\}$ converges non-tangentially to $1$, by the same token as before, {there exists $C_1>0$ such that, up to passing to a subsequence if necessary, for $n\in \mathbb N$ there exists a strictly increasing sequence $s_n\in (0,1)$ converging to $1$ so that} 
\[
k_\D(w_n, s_n)\leq C_1.
\] 
We claim that there exist $C_2>0$ and $n_0\in \mathbb N$ such that for every $n\geq n_0$ there exists $m_n\in \mathbb N$ so that
\begin{equation}\label{Eq:stima-close-geo}
k_\D(r_n, s_{m_n})\leq C_2.
\end{equation}
In particular, since $\{r_n\}$ converges to $1$ we have $\lim_{n\to \infty} s_{m_n}=1$. 
Assuming the claim for the moment, we have for $n\geq n_0$
\begin{equation*}
\begin{split}
k_\Omega(g(\zeta_n), g(w_{m_n}))&=k_\D(\zeta_n, w_{m_n})\leq k_\D(\zeta_n, r_n)+k_\D(r_n, s_{m_n})+k_\D(s_{m_n}, w_{m_n})\\& \leq C_0+C_2+C_1.
\end{split}
\end{equation*}
Therefore, by the Distance Lemma, 
\[
\lim_{n\to \infty}g(w_{m_n})=\lim_{n\to \infty}g(\zeta_n)=p,
\]
hence $p\in \Gamma(g;\{w_n\})$, that is, $\Gamma_{N}(g;\sigma)\subseteq \Gamma(g;\{w_n\})$

In order to prove \eqref{Eq:stima-close-geo}, note that, for every $n\in \mathbb N$,
\[
k_\D(s_n,s_{n+1})\leq k_\D(s_n,w_n)+k_\D(s_{n+1},w_{n+1})+k_\D(w_n,w_{n+1})\leq 2C_1+C=:C_2.
\]
Let {$n_0\in\mathbb N$} be such that $r_n\geq s_1$ for all $n\geq n_0$. For every $n$ there exists $m_n\in \mathbb N$ such that $r_n\in [s_{m_n}, s_{m_n+1})$. Since $[s_{m_n}, s_{m_n+1}]$ is a geodesic for $k_\D$ it follows that
\[
k_D(s_{m_n}, r_n)\leq k_D(s_{m_n}, s_{m_n+1})\leq C_2,
\]
and \eqref{Eq:stima-close-geo}.  
\end{proof}

In the sequel we will mainly use the previous result for univalent functions on $\Ha$. Using the Cayley transform from $\D$ to $\Ha$ one immediately has from Proposition~\ref{Prop:Lehto-NT}:

\begin{corollary}\label{Cor:non-tg-inH}
Let $g:\Ha \to \C$ be a univalent function. Let $\{w_n\}\subset\Ha$ be a sequence converging non-tangentially to $0$ such that there exists $C>0$ so that $k_\Ha(w_n,w_{n+1})\leq C$ for all $n$. Then
\[
\Gamma_{R}(g;0)=\Gamma_{N}(g;0)=\Gamma(g;\{w_n\}),
\]
where $\Gamma_{R}(g;0)$ is the cluster set of $\{g(r)\}_{r>0}$ at $0$, $\Gamma_{N}(g;0)$ is the non-tangential cluster set of $g$ at $0$ and $\Gamma(g;\{w_n\})$ is the cluster set of $\{g(w_n)\}$ for $n\to \infty$.
\end{corollary}

In the following we also need the following:
\begin{lemma}\label{lem:JWC}
Let $h:\Ha \to \Ha$ be holomorphic. Assume that $h$ extends $C^1$ at $0$ and $h(0)=0$. Then $h'(0)\in (0,+\infty)$. Moreover, if $h'(0)<1$ then $\{h^{\circ m}(w)\}$ converges non-tangentially to $0$. 
\end{lemma} 
\begin{proof}
The first part follows at once by the classical Julia-Wolff-Carath\'eodory's theorem for the unit disc (see, {\sl e.g.}, \cite[Theorem~1.7.3]{BCDM}) applied to the function $C\circ h\circ C^{-1}$  where $C:\D\to \Ha$ is defined by $C(z)=\frac{1-z}{1+z}$. If $h'(0)<1$, since $C$ is conformal at $1$ and maps $1$ to $0$, the result follows from \cite[Proposition~1.8.7]{BCDM} applied to $C\circ h\circ C^{-1}$.
\end{proof}

\section{Convex wedges with no Denjoy-Wolff property}\label{Sec:3}
Let $V \subset \Ha$ and $g:\mathbb H \to \C$ be a  holomorphic function. The \emph{graph} of $g$ over $V$ is 
\[
G(V, g):=\{(\zeta,w)\in\mathbb C^2: \zeta \in V, w=g(\zeta)\}\subset \C^2.
\]
As a matter of notation, if $(z_1,z_2,\ldots, z_N)\in \C^N$, $N\geq 3$, we denote by $z'':=(z_3,\ldots, z_N)$.

\begin{lemma}\label{lem:main}
	Let $V \subset \Ha$ be a simply connected domain such that $0\in\partial V$ and $\partial V$ is $C^{1+\varepsilon}$-smooth at $0$, for some $\varepsilon>0$.
	Suppose that $g:\mathbb H \to \C$ is a bounded univalent function. Let $\Omega$ be a  domain in $\mathbb C^N$, $N\geq 2$, such that
\begin{itemize}	
	\item[(i)] $\Omega \subset \{z\in \mathbb{C}^N: \Re z_1> 0\}$;	
	\item[(ii)] $G(V, g)\times\{0''\} \subset \Omega$. 
\end{itemize}		
		Then there exists a holomorphic self-map $F$ of $\Omega$ such that 
\begin{enumerate}
\item $F$ has no fixed points in $\Omega$,
\item $T(F)=\{0\}\times \Gamma_N(g;0)\times \{0''\}$.
\end{enumerate}	
In particular, if $g$ does not have non-tangential limit at $0$ then $\Omega$ does not have the Denjoy-Wolff property.		
\end{lemma}
\begin{proof}
Let $P:\C^n\to\C$ denote the projection $z \mapsto z_1$. By hypothesis (i),  $P(\Omega)\subset \mathbb H$. Let $\varphi:\mathbb H \to V$ be a biholomorphism. Since $0\in\partial V$ and $\partial V$ is assumed to be $C^{1+\varepsilon}$ at $0$, it follows that $0$ is an accessible point of $\partial V$ and, in particular, there exists a continous curve $\gamma:[0,1)\to V$ such that $\lim_{t\to 1}\gamma(t)=0$. Hence, (see, {\sl e.g.},\cite[Proposition~3.3.3]{BCDM}), there exists a point $\sigma\in \partial \Ha\cup \{\infty\}$ such that $\varphi$ has non-tangential limit $0$ at $\sigma$. Up to pre-composing $\varphi$ with an automorphism of $\Ha$, we can assume that $\sigma=0$.  Also, the hypothesis that $\partial V$ is $C^{1+\varepsilon}$-smooth at the origin {implies}, by the local version of the Kellogg-Warschawski theorem \cite[Theorem 1]{War}, {that} $\varphi$ extends $C^1$ at $0$. In particular, $\varphi(0)=0$ and $\varphi'(0)$ is  continuous at $0$. By Lemma~\ref{lem:JWC}, $\varphi'(0)>0$. 

Let $\lambda>0$ be such that $\lambda \varphi'(0) < 1$, let $h(z_1):=\varphi(\lambda z_1)$, $z_1\in \Ha$.  Note that $h\in \mathcal{O}(\mathbb H,\mathbb H)$,  $h$ is $C^1$ at $0$, $h(0)=0$ and $h'(0) <1$. Therefore again by Lemma~\ref{lem:JWC}, for every $\zeta\in \Ha$ the sequence $\{h^{\circ m}(\zeta)\}$  converges non-tangentially to $0$. Moreover, by construction, $h(\Ha)\subseteq V$.

Now, for $z\in \Omega$, let $f(z):=(h\circ P)(z)=h(z_1)$ and let 
\[
F(z):=(f(z),g(f(z)),0''), \quad z\in\Omega.
\]
 where $g:\mathbb H \to \mathbb C$ is the function in the hypothesis of the lemma.
		
By construction,
\[
F(\Omega)\subset G(V, g)\times\{0''\} \subset \Omega.
\]
Hence, by hypothesis (ii),  $F\in\mathcal{O}(\Omega,\Omega)$. Furthermore, for every $m\in \mathbb N$,
\[
F^{\circ m}(z_1,z_2,z'')=(h^{\circ m}(z_1),g(h^{\circ m}(z_1)),0'').
\]
Since  $h(z)$ has no fixed points, so does $F$. 

By construction, for every $z\in \Omega$, $\{h^{\circ m}(z_1)\}$ converges to $0$ non-tangentially as $m\to \infty$.  Since
\[
k_\Ha(h^{\circ m}(z_1), h^{\circ m+1}(z_1))\leq k_\Ha(z_1, h(z_1)), 
\]
it follows by Corollary~\ref{Cor:non-tg-inH} that $T(F)=\{0\}\times \Gamma_N(g;0)\times \{0''\}$.
\end{proof}

{\begin{remark}
The proof of Lemma~\ref{lem:main} works even if $g$ is not supposed to be bounded. However, if $g$ is not bounded, it might be that $\Gamma_N(g;0)=\{\infty\}$ (which just means that the iterates of $F$ escapes to $\infty$).
\end{remark}}

{The following lemma is probably well known but we give a sketch of the proof for the lack of a precise reference:}
\begin{lemma}\label{Lem:concave}
{ Let $a>0$. Let $f:[0,a]\to\R$ be a continuous  function such that $f(0)=0$ and $f(t)>0$ for $t\in (0,a]$. Then there exists a continuous strictly increasing concave function $\varphi:[0,+\infty)\to \R$ such that $\varphi(0)=0$ and $\varphi(t)\geq f(t)$ for all $t\in [0,a]$.}
\end{lemma}

\begin{proof}
Consider the family $\mathcal G$ of all real continuous concave functions $u$ on $[0,a]$ such that $u(t)\geq f(t)$ for all $t\in [0,a]$. Let $\tilde u(t):=\inf_{u\in \mathcal G}u(t)$ and let $\hat u$ be the lower semicontinuous regularization of $\tilde u$. Note that $\hat u\leq \tilde u$ and $\hat u$ is continuous.  Moreover, if $v$ is any lower semicontinuous function on $[t_0,t_1]\subseteq [0,a]$ such that $v(t)\leq \tilde u(t)$ for all $t\in [t_0,t_1]$ then $v(t)\leq \hat u(t)$ for all $t\in [t_0,t_1]$. With this at hand, it is a standard argument to check that $\hat u$ is concave by comparing with affine functions, and, also, that $\hat u\geq f$. Moreover, continuity of $f$ and $f(0)=0$ allow to see that given $\epsilon>0$ the function $[0,a]\mapsto \frac{\max f}{\delta}t+\epsilon$ belongs to $\mathcal G$ for $\delta>0$ such that $f(t)\leq \epsilon$ for all $t\in [0,\delta]$. Hence, $\hat u(0)=0$.

Now, recall that a non-constant concave function cannot have points of local minimum in the interior of the interval of definition. Let $M>0$ be the maximum  of $\hat u$ on $[0,a]$.  Since $\hat u$ does not have local minimum on $[0,a]$, there exists $0<t_0\leq a$  such that $\hat u(t)<M$ for all $t\in [0, t_0)$. Hence, $\hat u$ is strictly increasing in $[0, t_0]$: otherwise, there are  $0\leq s_0 < s_1\leq t_0$ such that $\hat u(s_0)\geq \hat u(s_1)$, implying that $\hat u|_{[s_0, t_0]}$ has a point of minimum in $(s_0, t_0)$ so it is constant, against  $\hat u(s_0)<\hat u(t_0)$. Finally, since $\hat u$ is concave, there exists $s\in (0,t_0)$---actually, infinitely many---such that $\hat u'(s)$ exists. Take into account that $\hat u$ is concave and strictly increasing, $\hat u'(s) > 0$, and, if we let $\varphi(t) := \hat u(t)$ for $t \in [0, s]$ and $\varphi(t) := \hat u'(s)(t - s) + \hat u(s)$ for $t>s$, it is easy to check that $\varphi$ satisfies the assertions of the lemma.
\end{proof}

{We also need the following lemma:}

\begin{lemma}\label{Lem:goodV}
{Let $V\subset \Ha$ be a  convex domain such that $0\in \partial V$ and $\partial V$ is $C^{1+\varepsilon}$ at $0$ for some $\varepsilon>0$. Then there exists a  convex domain $V'\subseteq V$ such that $\partial \Ha\cap \partial V'=\{0\}$ and $\partial V'$ is $C^{1+\varepsilon}$ at $0$.
}
\end{lemma}
\begin{proof}
If  $\partial \Ha\cap \partial V=\{0\}$ we can take $V'=V$. If $\partial \Ha\cap \partial V=i [-\alpha, \beta]$ for some $\alpha, \beta>0$, we can take $V':=V\cap \{\zeta\in \C: |\zeta-1|<1\}$. If $\partial \Ha\cap \partial V=i [-\alpha, 0]$ for some $\alpha>0$, by definition of $C^{1+\varepsilon}$-smooth boundary, there exist $\delta>0$ and a homeomorphism $\omega:(-1,1)\to \partial V\cap D(0,\delta)$, where $D(0,\delta)=\{\zeta\in \C: |\zeta|<\delta\}$, such that $\omega(t)=\omega_1(t)+i\omega_2(t)$ with $\omega_1, \omega_2$ real and $C^1$-smooth functions on $(-1,1)$ with $\varepsilon$-H\"older continuous first derivatives. We can also assume that $\omega((-1,0))=i[-\alpha,0]\cap D(0,\delta)$ and $\omega_1(t), \omega_2(t)>0$ for $t\in (0,1)$---that is, $\omega((0,1))$ lies in the first quadrant. Define $\tilde \omega(t):=\omega(t)$ for $t\in [0,1)$ and $\tilde \omega(t)=\omega_1(-t)-i\omega_2(-t)$ for $t\in (-1,0)$. Taking into account that $\omega_1(t)\equiv 0$---and hence $\omega_1'(t)\equiv 0$---for $t\in (-1,0]$, it is not difficult to see that $\tilde \omega((-1,1))$ is a $C^{1+\varepsilon}$-smooth Jordan arc. In fact, since we are assuming that $\omega((0,1))$ lies in the first quadrant, $\tilde \omega((-1,1))$ is the reflection of the curve $\omega([0,1))$ across the real axis.  Finally, note that we can choose  $s>0$  so small that $\tilde\omega([-s,s])\subset V$. Let $L$ be the segment connecting $\tilde\omega(s)$ to  $\tilde\omega(-s)$. It follows by construction that $\tilde\omega([-s,s])\cup L$ is a Jordan curve which bounds a bounded convex domain $V'\subset V$. Moreover, by construction,  $\partial \Ha\cap \partial V'=\{0\}$ and $\partial V'$ is $C^{1+\varepsilon}$ at $0$. The case $\partial \Ha\cap \partial V=i [0,\alpha]$ for some $\alpha>0$ is similar and we omit details.
\end{proof}

Now we are in a good to shape to prove our main results.

\begin{proof}[Proof of Theorem \ref{thm:main}.]
{According to Lemma~\ref{Lem:goodV}, and since $W(V',\Psi, \delta)\subseteq W(V,\Psi, \delta)$ for any  convex domain $V'\subset V$ such that $0\in \partial V'$, we can assume that $V$ 
\begin{equation}\label{Eq:V-good-0}
\partial V\cap \partial \Ha=\{0\}.
\end{equation}}

Fix $\delta>0$. Let $a>0$ be such that 
\begin{equation}\label{Eq:espilon-small}
\frac{9}{4}a^2<\delta^2.
\end{equation}
For $n\in \mathbb{N}$ define 
\[
S_{2n}:=\{\zeta \in \mathbb C: \Re \zeta =\frac{a}{2n}, \Im \zeta  \in [-\frac{a}{2},\frac{a}{6}]\},
\]
and 
\[
S_{2n+1}:=\{\zeta \in \mathbb C: \Re \zeta=\frac{a}{2n+1}, \Im \zeta \in [-\frac{a}{6},\frac{a}{2}]\}.
\]
Let {$S:= \cup_{n=2}^\infty S_{n}$} and let 
 \begin{equation}\label{eqn:ourdomain}
Q:=\{\zeta \in \mathbb C: {0<\Re \zeta < a,} -\frac{a}{2}<|\Im \zeta| <\frac{a}{2}\}\setminus S.
\end{equation}

\begin{center}\label{Fig1}
	\begin{tikzpicture}[scale=6]
		
		% Draw boundary of the horizontal strip
		\draw[thick] (0, -0.5) -- (0, 0.5);
		\draw[thick] (1, -0.5) -- (1, 0.5);
		\draw[thick] (0, 0.5) -- (1, 0.5);
		\draw[thick] (0, -0.5) -- (1, -0.5);
		
		% Draw a few slits T_n
		\foreach \n in {1,2,3,4,5,6,7}{
			\pgfmathsetmacro{\xeven}{1/(2*\n)}
			\pgfmathsetmacro{\xodd}{1/(2*\n-1)}
			% T_{2n}
			\draw[black, thick] (\xeven, -0.5) -- (\xeven, 1/6);
			% T_{2n+1}
			\draw[black, thick] (\xodd, -1/6) -- (\xodd, 0.5);
		}
		
		% Optional axis labels
		%\node at (0.5, 0.55) {\small$\operatorname{Im}(z)$};
		%\node[rotate=90] at (-0.05, 0) {\small$\operatorname{Re}(z)$};
		\node at (0.5, -0.6) {The domain $Q$.};
		\node at (0.04, 0) {\tiny $\cdots$};
	\end{tikzpicture}
\end{center}
Note that $Q$ is a bounded simply connected domain. Let $h:\D\to Q$ be a Riemann map.  {Fix $y\in [-\frac{a}{6},\frac{a}{6}]$ }and let 
\[
C_n^y:=\{x+iy: \frac{a}{2n+1}\leq x \leq \frac{a}{2n}\}.
\]
It is easy to see that $(C_n^y)$ is a null-chain in $Q$ and that $(C_n^y)$ is equivalent to $(C_n^{y'})$ for all {$y, y'\in [-\frac{a}{6},\frac{a}{6}]$.}  The prime end defined by such $(C_n^y)$ corresponds via $h$ to a point $\sigma\in\partial \D$, and we can assume that $\sigma=1$. {Using such null-chains, one can easily verify that $[-\frac{ia}{6},\frac{ia}{6}]\subseteq \Pi[\hat{h}(\underline{1})]$. The converse inclusion can be proved using circular null-chains (see, {\sl e.g.}, \cite[Proposition~4.1.1]{BCDM}) showing that if $p\in [-\frac{ia}{2},\frac{ia}{2}]\setminus[-\frac{ia}{6},\frac{ia}{6}]$ then $p\not\in \Pi[\hat{h}(\underline{1})]$. For instance, let $\gamma\in (\frac{a}{6}, \frac{a}{2}]$. If $\gamma i\in  \Pi[\hat{h}(\underline{1})]$ then one can find a null-chain $(G_n)$ equivalent to $(C_n^y)$ defined as follows: there exists a strictly decreasing sequence of positive real numbers $\{r_n\}$ converging to $0$ such that, if $C(i\gamma, r_n)=\{\zeta\in \C: |\zeta-i\gamma|=r_n\}$, then $G_n$ is a connected component of $C(i\gamma, r_n)\cap Q$. Let $w_0\in Q$. If $n$ is sufficiently large so that $r_n<\gamma-\frac{a}{2}$, it is simple to see that $\overline{G_n}$ is a Jordan arc which divides $Q$ into two connected components, one of them always contains $w_0$ and the other is contained in the ``channel'' defined by $S_{2m_n}$ and $S_{2(m_n+1)}$ for a suitable $m_n\in \mathbb N$ so that $\{m_n\}$ converges to $0$ as $n\to \infty$. Therefore, $\{G_n\}$ cannot define a null-chain, hence $i\gamma\not\in \Pi[\hat{h}(\underline{1})]$.}

By Proposition~\ref{Prop:Lehto-NT},  $\Gamma_N(h;1)=[-\frac{ia}{6},\frac{ia}{6}]$. One can also easily check that the impression of the prime end $\hat{h}(\underline{1})$ is given by $[\frac{-ia}{2}, \frac{ia}{2}]$ hence (see, {\sl e.g.}, \cite[Proposition~4.4.4]{BCDM}) the cluster set of $h$ at $1$ is given by $[\frac{-ia}{2}, \frac{ia}{2}]$. In particular, $\lim_{\zeta\to 1}\Re h(\zeta)=0$.

Let $C:\D\to \Ha$ be the Cayley transform defined as $C(z)=\frac{1-z}{1+z}$. Let $g:=h\circ C^{-1}$. Hence, $g$ is a biholomorphism from $\Ha$ to $Q$. Since $C$ is conformal at $1$ and $C(1)=0$ we have
\begin{equation}\label{Eq:nt-segment}
\Gamma_N(g;0)=[-\frac{ia}{6},\frac{ia}{6}].
\end{equation}
Moreover, $\lim_{\zeta\to 0}\Re g(\zeta)=0$.
Let 
\[
V_a:=V\cap \{\zeta\in \C: |\zeta|<a\}.
\]
{Note that} $V_a\subset\Ha$ is a bounded convex domain, { $\partial V_a\cap \partial \Ha=\{0\}$ by \eqref{Eq:V-good-0}}, and $\partial V_a$ is $C^{1+\varepsilon}$-smooth at $0$. Moreover, by \eqref{Eq:espilon-small}, for every $\zeta\in V_a$ we have
\[
|\zeta|^2+|g(\zeta)|^2=|\zeta|^2+(\Re g(\zeta))^2+(\Im g(\zeta))^2\leq a^2+a^2+\frac{a^2}{4}<\delta^2.
\]
{ For $t\in (0,a]$ let $V_{a,t}:=\{\zeta\in V_a, \Re \zeta < t \}$. Note that $\Re g$  extends continuously on $\overline{V_{a,t}}$, since it is harmonic on $\partial V_a\setminus\{0\}\subset\Ha$ and it is continuous at $0$ because $\lim_{\zeta\to 0}\Re g(\zeta)=0$.}

{ Therefore, $[0,a]\ni t\mapsto \Phi(t):= \max\{\Re g(\zeta): \zeta\in \overline{V_{a,t}}\}$ is  a continuous non-decreasing function on $[0,a]$, $\Phi(0)=0$ and $\Phi(t)\leq a$ for all $t\in (0,a]$}.

{ Hence, by Lemma~\ref{Lem:concave},} one can find a continuous concave strictly increasing function  $\hat \Phi:[0,\infty) \to [0,\infty)$ such that $\Phi(t)< \hat \Phi(t)$ for all {$t\in(0,a]$} and $\hat \Phi(0)=0$.  Let $\Psi_0(t):=\hat \Phi^{-1}(t)$. 

Now, by construction, $\|(z_1, g(z_1))\|<\delta$ for every $z_1\in V_a$. Moreover, for every $z_1\in V_a$ we have
\[
\Re g(z_1)\leq \Phi(\Re z_1)< \hat \Phi(\Re z_1),
\]
that is,
\begin{equation}\label{Eq:stay-below-g}
\Re z_1 > \Psi_0(\Re g(z_1)).
\end{equation}
Hence, $G(V_a,g)\subset W(V,\Psi_0,\delta)$. The result follows then by \eqref{Eq:nt-segment} and Lemma~\ref{lem:main}.
\end{proof}

{Before proving Theorem~\ref{cor:existence-nonDW} we need a technical lemma.  While it is probably well-known, we sketch its proof due to a lack of direct references}.

\begin{lemma}\label{Lem:convex-smooth-below}
{Let $f:[0,\infty)\to [0,\infty)$ be a continuous strictly increasing convex function such that $f(0)=0$. Then there exists a $C^\infty$-smooth, strictly increasing, strongly convex function $\varphi:(0,\infty)\to (0,\infty)$ such that $\varphi(t)<f(t)$ for all $t>0$ and $\lim_{t\to 0^+}\varphi^{(m)}(t)=0$ for all $m=0,1,\ldots$.}
\end{lemma}
\begin{proof}
By standard results in convex analysis,  both the left derivative $f'_-(t)$ and the right  derivative $f_+'(t)$  exist for all $t> 0$, as well as $f_+'(0)$. Furthermore, since $f$ is strictly increasing, $0\leq f'_+(0)<f'_-(s)\leq f'_+(s)<f'_-(t)\leq f_+'(t)$ for all $0<s<t$.  Also, the set of points $t\geq 0$ such that $f'_-(t)<f_+'(t)$ (that is, where $f'$ does not exist) is countable. In particular, $[0,+\infty)\ni t\mapsto f'_-(t)$ is a strictly increasing function. 

Let $\{t_n\}$ be a strictly decreasing sequence in $(0,1]$ converging to $0$, with $t_1=1$. For each $n=1,\ldots$, let $\alpha_n:=\min\{\frac{f'_-(t_{n+2})}{2}, e^{-\frac{1}{t_{n+2}}}\}$. Note that $\{\alpha_n\}$ is a strictly decreasing sequence of positive real numbers converging to $0$ and $f'_-(t)>\alpha_n$ for $t\geq t_{n+2}$. Let $\beta_n:=\alpha_n-\alpha_{n+1}$. Note that $\sum_{j\geq n}\beta_j=\alpha_n$. For every $n=1,\ldots$, let $\chi_n$ be a bump function with compact support in an open interval $B_n$ such that $B_1\subset (t_2, 2)$, $B_n\subset (t_{n+1}, t_{n-1})$, $t_n\in B_n$ for $n\geq 2$ and  $\int_0^1\chi_n(s)ds=\beta_n$ for $n\geq 1$. We can choose the intervals $B_n$'s in such a way that every $t\in (0,1)$ belongs to at least one and at most a finite number of such $B_n$'s. Furthermore, we can choose the supports of the $\chi_n$ in such a way that for every $t\in (0,1)$ there is $n$ so that $\chi_n(t)>0$. Note that $\int_0^{t_{n+1}} \chi_j(s)ds=0$ for $j<n$ and $\int_0^{t_{n+1}}\chi_j(s)ds\leq \beta_j$ for $j\geq n$.  Let $\chi(t):=\sum_{n\geq 1}\chi_n(t)$. By construction, $\chi$ is well defined, $C^\infty$-smooth and $\chi(t)>0$ for all $t\in (0,1)$. Let $b(t):=\int_0^t \chi(s)ds$. Note that $b$ is $C^\infty$-smooth and strictly increasing. Now, let $t\in (0,t_2]$. Hence, there exists $n$ such that $t_{n+2}<t\leq t_{n+1}$. Therefore, 
 \begin{equation*}
 \begin{split}
b(t)&=\int_0^t \chi(s)ds\leq \int_0^{t_{n+1}} \chi(s)ds=\sum_{j\geq 1}\int_0^{t_{n+1}}\chi_j(s)ds=\sum_{j\geq n}\int_0^{t_{n+1}}\chi_j(s)ds\\&\leq \sum_{j\geq n}\beta_j=\alpha_{n}.
 \end{split}
 \end{equation*}
Hence, $b(t)< f'_-(t_{n+2})<f'_-(t)$. Moreover, for every fixed $m=0,1,2,\ldots$ and $t_{n+2}<t\leq t_{n+1}$
 \[
0<\frac{b(t)}{t^m}\leq \frac{\alpha_n}{t^m}\leq \frac{e^{-\frac{1}{t_{n+2}}}}{t_{n+2}^m}\to 0 \quad \hbox{as\ }t\to 0^+.
 \]
Therefore, $b:[0,t_2]\to\R$ is a strictly increasing $C^\infty$-smooth function such that $b(t)< f'_-(t)$ for all $t\in (0,t_2]$, $b(0)=0$ and the right derivatives of $b$ of any order at $0$ are $0$. We can clearly extend $b$ as a $C^\infty$-smooth strictly increasing function on all $[0,+\infty)$ such that $b(t)< f'_-(t_2)<f'_-(t)$ for all $t\geq t_2$. Hence, for all $t>0$,
\[
\varphi(t):=\int_0^t b(s)ds<\int_0^t f'(s)ds=f(t).
\]
By construction, $\varphi$ satisfies all the assertions of the lemma.
\end{proof}

\begin{proof}[Proof of Theorem~\ref{cor:existence-nonDW}]
Let $a=1$ and let $Q$ be defined in \eqref{eqn:ourdomain}. Let $\Psi_0:[0,\infty)\to  [0,\infty)$ be a continuous strictly increasing convex function which satisfies \eqref{Eq:stay-below-g} and $\Psi_0(0)=0$. {By Lemma~\ref{Lem:convex-smooth-below}, we can find a $C^\infty$-smooth, strictly increasing, strongly  convex function $\Psi:(0,\infty)\to (0,\infty)$ such that $\lim_{t\to 0^+}\Psi^{(m)}(t)=0$ for all $m=0,1,\ldots$ and $\Psi(t)<\Psi_0(t)$ for all $t>0$. We extend $\Psi$ to an even $C^\infty$-smooth convex function on $\R$ such that $\Psi^{(m)}(0)=0$ for $m=0,1,\ldots$ by setting $\Psi(-t)=\Psi(t)$ for $t<0$ and $\Psi(0)=0$.} In particular, $\Psi$ is strongly convex outside $0$. The domain
\[
\Omega_0:=\{(z_1,z_2)\in\C^2: \Re z_1 > \Psi(\Re z_2)\}
\]
is an unbounded smooth convex domain,  $(0,0)\in \partial\Omega_0$. The defining function of $\Omega_0$ is given by $\rho(z_1,z_2)=\Psi(\frac{1}{2}(z_2+\overline{z_2}))-\frac{1}{2}(z_1+\overline{z_1})$. For every $p=(p_1,p_2)\in \partial \Omega_0$, the complex tangent space $T_p^\C \partial \Omega_0$ is given by
\[
T_p^\C \partial \Omega_0=\{(v_1,v_2)\in \C^2: v_1=\Psi'(\Re p_2)v_2\}.
\]
The Levi form $\mathcal L_p: T_p^\C \partial \Omega_0\to \R$ at $p$ is given by 
\[
\mathcal L_p(v_1,v_2)=\Psi''(\Re p_2) |v_2|^2/4.
\]
 
Since $\Psi''(t)>0$ for all $t\neq 0$, it follows that  $\partial \Omega_0$ is {strongly pseudoconvex} for all $p\in \partial \Omega_0$ such that $\Re p_2\neq 0$. While, for $p\in \partial \Omega_0$ such that $\Re p_2=0$,  $H:=T_p^\C\partial\Omega_0=\{z\in \C^2: z_1=0\}$ is a complex supporting hyperplane for $\Omega_0$. Since  $\partial \Omega_0\cap H=\{(0,it): t\in\R\}$, it follows that  $\partial\Omega_0$ does not have non-trivial analytic discs on the boundary. 

Now, let $\tilde\Omega:=\Omega_0\cap \{(z_1,z_2)\in \C^2: \|(z_1,z_2)\|<10\}$. Then, $\tilde\Omega$ is a bounded convex domain, with no non-trivial analytic discs on the boundary and $\partial\tilde\Omega$ is $C^\infty$ smooth in a neighborhood of $(0,0)$. Let $\Omega$ be a convex smooth regularization of $\tilde\Omega$. By construction, $\Omega$ is a bounded convex domain with smooth boundary and $\partial \Omega$ does not contain non-trivial analytic discs. Moreover, since $\Psi\leq \Psi_0$, $\Omega$ contains the graph of the function $g$ (defined in the proof of Theorem~\ref{thm:main}) over $\{\zeta\in \Ha: |\zeta|<1\}$. Therefore, $\Omega$ does not have the Denjoy-Wolff property. 
\end{proof}

\section{Target set on convex domains with non-trivial analytic discs on the boundary}\label{Sec:4}

In this section we deal with convex domains with non-trivial analytic discs on the boundary.

 \begin{proof}[Proof of Theorem~\ref{prop:target-set-discs}]
Up to invertible complex affine transformations, we can assume that 
\begin{itemize}
\item[(i)] $\Omega\subset \{z\in\mathbb C^n:\Re (z_1)>0\}$,
\item[(ii)] $\partial\Omega$ contains the disc $\Delta:= \{0\}\times 4\D \times \{0''\}$.
\end{itemize}

Let $D:= \Omega \cap \{(z_1,z')\in \C^N: z'=0\}$---we can consider $D$ as a bounded convex domain in $\C$ with the $z_1$-variable. Since $\Omega$ is a bounded convex domain with $C^{1+\varepsilon}$-smooth boundary it follows that $D$ is also a bounded (planar) convex domain with $C^{1+\varepsilon}$-smooth  boundary and $0\in \partial D$. 

Let $V:=\{\zeta\in \C: 2\zeta\in D\}\subset D$. Note that $0\in \partial V$ and $\partial V$ is $C^{1+\varepsilon}$-smooth. We claim that 
\begin{equation}\label{Eq:product contained}
V\times 2\D\times \{0''\}\subset\Omega.
\end{equation}
Indeed, for every $z_1\in D$, $q\in 4\D$ and $t\in (0,1)$ we have $(tz_1,(1-t)q,0'')\in \Omega$. Fix $\zeta\in V$ and $\eta\in 2\D$. Then taking $z_1=2\zeta\in D$, $q=2\eta\in 4\D$ and $t=\frac{1}{2}$ we see that $(\zeta,\eta,0'')\in \Omega$.

Therefore, by \eqref{Eq:product contained}, if $g:\Ha \to \D$ is any univalent function, we have that {$G(V,g)\times\{0'\}\subset \Omega$}. It follows by Lemma~\ref{lem:main} that there exists a  holomorphic self-map $F$ of $\Omega$ such that 
 $T(F)=\{0\}\times \Gamma_N(g;0)\times \{0'\}$. From this {and Proposition~\ref{Prop:Lehto-NT}} the statement follows at once.
 \end{proof}

\begin{remark}
It is clear from the proof of {Theorem}~\ref{prop:target-set-discs} that the argument applies as soon as there exists a point $p$ in the interior of $\Delta$ such that $\partial\Omega$ is $C^{1+\varepsilon}$ at $p$, or, more generally, if there exists a complex affine line $L$ such that $L$ contains a point in the interior of $\Delta$, $D:=L\cap \Omega$ is non-empty and $\partial D$ is $C^{1+\varepsilon}$-smooth at $p$.
\end{remark}

\begin{example}\label{ex:example1}
Consider the domain $D:=\D\setminus\gamma([1,\infty))$, where $ \gamma : [1, \infty) \to \mathbb{C} $ is the spiral curve defined by $\gamma(t) = \left(1 - \frac{1}{t}\right) e^{2\pi i t}, \quad t \geq 1.$

\begin{center}
	\begin{tikzpicture}
		\begin{axis}[
			hide axis,
			unit vector ratio=1 1 1,
			xmin=-1.1, xmax=1.1,
			ymin=-1.1, ymax=1.1,
			width=9cm,
			axis equal image,
			clip=true
			]
			
			% Unit circle
			\addplot [
			domain=0:360,
			samples=800,
			smooth,
			ultra thick,
			black
			] ({cos(x)}, {sin(x)});
			
			% Smooth spiral slit
			\addplot [
			domain=1:60,       % you can increase this to 80?100 for more turns
			samples=6000,      % high resolution for smoothness
			smooth,
			thick,
			black
			] (
			{(1 - 1/x)*cos(deg(2*pi*x))},
			{(1 - 1/x)*sin(deg(2*pi*x))}
			);
			
		\end{axis}
	\end{tikzpicture}
	
	\vspace{0.5em}
	\text{The domain $D$.}
\end{center}

The image $ \gamma([1, \infty)) \subset \D $ is a simple spiral that winds infinitely many times around the origin and accumulates everywhere on the unit circle. This domain $D\subset \D$ is a bounded simply connected domain and considering the crosscuts of this domain that lie in the line segments joining the origin to the unit circle shows that there exists a prime end $\underline{x}$ such that $\Pi[\underline{x}]=\partial\D$. Then, it follows by {Theorem}~\ref{prop:target-set-discs} that if $\Omega\subset\C^N$, $N\geq 2$ is any bounded convex domain with $C^{1+\varepsilon}$-smooth boundary and $\partial \Omega$ contains a complex affine disc, there exists a holomorphic self-map $F$ of $\Omega$ such that $T(F)$ is a circle in $\Delta$.
\end{example}

\begin{example}
Let $E \subset \D$ be a simply connected domain bounded by the von Koch snowflake \cite{Koc}. It is a simply connected domain bounded by a nowhere differentiable Jordan curve of Hausdorff dimension $\log 4/\log 3$, in particular by Carath\'eodory extension theorem, the Riemann map $f:\D\to E$ extends to a homeomorphism from $\ov \D$ to $\overline{\mathbb E}$.

%\vspace{\baselineskip}
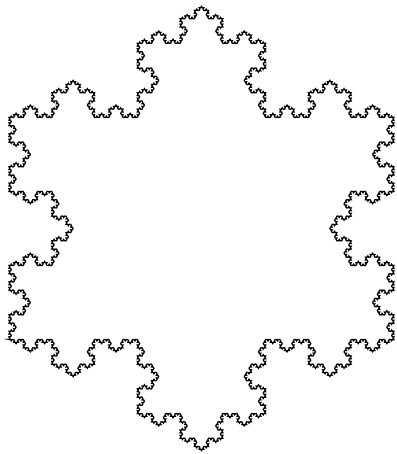
\begin{center}
	\begin{tikzpicture}[scale=0.2] % Reduced scale to keep the size consistent
		\draw[line width=0.5pt, l-system={rule set={F -> F-F++F-F}, axiom=F++F++F, order=5, angle=60, step=3pt}]
		lindenmayer system -- cycle;
	\end{tikzpicture}
	
	\vspace{0.5em}
	\text{The domain $E$ bounded by the Koch snowflake.}
\end{center}

Let $h:\D\to D$ (where $D$ is defined in Example~\ref{ex:example1}) be a Riemann map. Then $f\circ h$ is a bounded univalent map which has a prime end whose impression is the von Koch snowflake. Again by {Theorem}~\ref{prop:target-set-discs}, if $\Omega\subset\C^N$, $N\geq 2$ is any bounded convex domain with $C^{1+\varepsilon}$-smooth boundary and $\partial \Omega$ contains a complex affine disc, there exists a holomorphic self-map $F$ of $\Omega$ such that $T(F)$ is a von Koch snowflake. In particular, $T(F)$ has non-integer Hausdorff dimension.
\end{example}

\end{document}